\documentclass[a4paper,12pt,intlimits,oneside]{amsart}

\usepackage{enumerate}
\usepackage{amsfonts}
\usepackage{amsmath,amsthm,amssymb}

\newcommand{\bC}{{\mathbb C}}
\newcommand{\bN}{{\mathbb N}}
\newcommand{\bR}{{\mathbb R}}

\newcommand{\bS}{{\mathbb S}}
\newcommand{\bB}{{\mathbb B}}

\def\pf{\proof}
\def\bpf{\pf}
\newcommand{\epf}{ $\Box$\medskip}
\def\A{{\mathcal A}}

\def\H{{\mathcal H}}
\def\K{{\mathcal K}}
\def\O{{\mathcal O}}

\def\ss{\subseteq}

\def\BMO{B\! M\! O}

\def\BMOA{B\! M\! O\! A}
\def\LMOA{L\! M\! O\! A}

\def\dst{\displaystyle}
\def\eps{\varepsilon}
\def\ov{\overline}
\def\p{\partial}

\def\lsim{\raisebox{-1ex}{$~\stackrel{\textstyle <}{\sim}~$}}
\def\gsim{\raisebox{-1ex}{$~\stackrel{\textstyle >}{\sim}~$}}

\def\lux{\rm lux}

\newcounter{rea}
\newcommand{\comment}[1]{}
\newcounter{res}
\newtheorem{thm}{Theorem}[section]
\newtheorem{prop}[thm]{Proposition}
\newtheorem{cor}[thm]{Corollary}
\newtheorem{lem}[thm]{Lemma}
\newtheorem{defn}[thm]{Definition}
\newtheorem{remark}[thm]{Remark}
\begin{document}

\title[Hardy-Orlicz spaces]{Hankel operators and weak factorization \\ for Hardy-Orlicz
spaces}

\author[A. Bonami]{Aline Bonami}
\address{MAPMO-UMR 6628,
D\'epartement de Math\'ematiques, Universit\'e d'Orleans,
45067 Orl\'eans Cedex 2, France}
\email{{\tt Aline.Bonami@univ-orleans.fr}}
\author[S. Grellier]{Sandrine Grellier}
\address{MAPMO-UMR 6628,
D\'epartement de Math\'ematiques, Universit\'e d'Orleans,
45067 Orl\'eans Cedex 2, France}
\email{{\tt Sandrine.Grellier@univ-orleans.fr}}


\begin{abstract}
We study the holomorphic Hardy-Orlicz spaces $\H^{\Phi}(\Omega)$,
where $\Omega$ is the unit ball or, more generally,  a convex
domain of finite type or a strictly pseudoconvex domain in
$\bC^n$. The function $\Phi$ is  in particular such that
$\H^{1}(\Omega)\subset \H^{\Phi}(\Omega)\subset \H^{p}(\Omega)$
for some $p>0$. We develop for them maximal characterizations,
atomic and molecular decompositions. We then prove weak
factorization theorems involving the space $\BMOA(\Omega)$.

As a consequence, we  characterize those Hankel operators which
are bounded from  $\mathcal H^\Phi(\Omega)$ into $\mathcal
H^1(\Omega)$.

\end{abstract}

\dedicatory{ This paper is dedicated to the memory of Andrzej Hulanicki who was a colleague, a friend we will never forget.}

\maketitle

\section*{Introduction}

The following work has been motivated by a new kind of
factorization in the unit disc, obtained  in \cite{BIJZ}.  Namely,
the product of a function in $\BMOA$ with a function in the Hardy
space $\H^1$ of holomorphic functions lies in  some
Hardy-Orlicz space defined in terms of the function
$\Phi_1(t):=\frac t{\log(e+t)}$. Conversely,  every holomorphic
function of this Hardy-Orlicz space can be  written as the product
of a function in $\BMOA$ and a function in $\H^1$. This exact
factorization relies heavily on the classical factorization
theorem through Blaschke products and
cannot generalize in higher dimension. On the other hand, it has
been proven by Coifman, Rochberg and Weiss in the seventies
\cite{CRW} that $\H^p$, for $p\leq 1$, admits weak factorization, namely, $F=\sum_jG_jH_j$ with $\sum_j\|G_j\|_q^p\|\|H_j\|_r^p\leq
C_{p q}\|F\|_p^p$ when $\frac 1q+\frac 1r=\frac 1p$. This has been extended later on by Krantz and Li  to strictly
pseudo-convex domains \cite{KL2}, then by Peloso, Symesak and the
authors of the present paper to convex domains of finite type
\cite{BPS}, \cite{GP}. We rely on the methods of these two last papers, which are
somewhat simpler, to obtain the weak factorization of Hardy-Orlicz
spaces under consideration. Note that such a weak factorization
for $\H^p$ is typical of the case $p\leq 1$, in opposition to the
case of the unit disc.

A natural application of such factorizations is the
characterization of classes of symbols of Hankel operators. We
are able to characterize symbols of Hankel operators mapping
continuously  Hardy-Orlicz spaces  into $\H^1$ for a large class of Hardy-Orlicz spaces  containing $\H^1$.  We do this for  all  domains for which we have
weak factorization. However weak factorization is a stronger
property, since the Hardy-Orlicz spaces under consideration are not
Banach spaces. We have given in \cite{BGS} a direct proof of the
fact that Hankel operators are bounded on $\H^1$ in the unit ball
if and only if their symbol is in the space $\LMOA$, without
involving Hardy-Orlicz spaces, even if the idea of weak
factorization indirectly  is present in this Note.

Let us mention, in the same direction, the factorization obtained
by Cohn and Verbitski in the disc \cite{CV}, which allows to
characterize those symbols for which the corresponding Hankel
operator is bounded from $H^2$ into some Hardy-Sobolev space. The
generalization in higher dimension of their factorization seems
much more difficult than ours.\\

At the end of this paper, we state  the same theorems for a class of domains in $\bC^n$, which includes the strictly pseudoconvex domains and the convex domains of finite type. We explain rapidly how to modify the proofs.

\bigskip

Let us give some notations and describe more precisely the
results. Let $\bB^n$ be the unit ball and $\bS^n$ be the unit
sphere in $\bC^n$. Let $\Phi$ be a continuous, positive and
non-decreasing function on $[0,\infty)$. The Hardy-Orlicz space
$\H^{{\Phi}}(\bB^n)$ is defined  as the space of holomorphic
functions $f$ so that
\begin{equation}\label{hardy-orlicz}
  \sup_{0<r<1} \int_{\bS^n}
{\Phi}(|f(rw)|) \,d\sigma(w) <\infty
\end{equation}
where $d\sigma$  denotes the surface measure on $\bS^n$. We
recover Hardy spaces $\mathcal H^p(\bB^n)$ when $\Phi(t)=t^p$. We
are especially interested in the case
$\Phi_p(t)=\frac{t^p}{\log(e+t)^p}$, $0<p\le 1$ since the space
$\H^{\Phi_p}(\bB^n)$ arises naturally in the study of pointwise
product of functions in $\mathcal H^p(\bB^n)$ with functions in
$\BMOA(\bB^n)$ inside the unit ball. Indeed, we prove that the product of
an $\mathcal H^p(\bB^n)$-function and of a $\BMOA(\bB^n)$-function
belongs to the space $\H^{{\Phi_p}}(\bB^n)$ and, conversely, that
there is weak factorization.

 We will
restrict  to concave functions $\Phi$, which satisfy an additional
assumption so that
$\H^1(\bB^n)\subset\H^{{\Phi}}(\bB^n)\subset\H^{p}(\bB^n)$ for
some $0<p\le 1$. In particular, any function $f$ in the Orlicz
space $\H^{{\Phi}}(\bB^n)$ admits a unique boundary function still
denoted by $f$ which, by Fatou Theorem, satisfies $\dst
\int_{\bS^n}\Phi(|f|)d\sigma<\infty$.

We also consider the real Hardy-Orlicz space $H^{\Phi}(\bS^n)$
defined as the space of distributions on $\bS^n$ which have an
atomic decomposition. More precisely, $H^{\Phi}(\bS^n)$ is the set
of distributions $f$ which can be written as $\sum_{j=0}^{\infty}
a_j$, where the $a_j$'s satisfy adapted cancellation properties,
are supported in some ball $B_j$ and are such  that $\sum_j
\sigma(B_j){\Phi}(\Vert a_j\Vert_2\sigma(B_j)^{-\frac 12})<\infty$.

We first prove usual maximal characterizations of  Hardy-Orlicz
spaces. As a corollary, we obtain that the Hardy-Orlicz space
$\H^{\Phi}(\bB^n)$ continuously embeds into $H^{\Phi}(\bS^n)$,
while the Szeg\"o projection is a projection onto
$\H^{\Phi}(\bB^n)$. In particular, every $f\in\H^{\Phi}(\bB^n)$
has boundary values that belong to $H^{\Phi}(\bS^n)$, and $f$ may
be written in terms of the Szeg\"o projection of its atomic
decomposition.  The work of Viviani \cite{Viviani} plays a central role:
atomic decomposition is proved there for Hardy-Orlicz spaces in
the context of spaces of homogeneous type with a restriction on the lower type $p$ of $\Phi$, which, in the case of the unit ball, is the condition $p>\frac{2n}{2n+1}$. We prove the atomic decomposition for all values of $p$, with the same kind of control of the norm as the one obtained by Viviani.

Since the Szeg\"o
projection of an atom is a molecule, we also get a molecular
decomposition as in the classical Hardy spaces (\cite{TW} for
instance).

The atomic decomposition allows to prove a (weak) factorization
theorem on $\H^{\Psi}(\bB^n)$, which coincides with the one for
$\H^p(\bB^n)$ when $\Psi(t)=t^p$. In particular, we generalize the
factorization theorem proved in the disc for $\mathcal H^{\Phi_1}$
in \cite{BIJZ}. More precisely, we prove that, given any
$f\in\H^{\Psi}(\bB^n)$ there exist $f_j\in\H^{\Phi}(\bB^n),g_j\in
\BMOA(\bB^n)$ such that $f=\sum_{j=0}^{\infty}f_jg_j$ where $\Psi$
and $\Phi$ are linked by the relation $\Psi(t)=\Phi\left(\frac
t{\log(e+t)}\right)$.

 As a consequence, we  characterize the class of symbols for which the Hankel
 operators are bounded from $\mathcal H^{\Phi}(\bB^n)$ to $\mathcal H^1(\bB^n)$.
 They belong to the dual space of $\mathcal H^{\Psi}(\bB^n)$, which can be
  identified with the $\BMOA$-space with weight $\rho_\Psi$ where $\rho_\Psi(t)=\frac 1{t\Psi^{-1}(1/t)}$. Weighted $\BMOA$-spaces have been considered by Janson in the  Euclidean space \cite{Janson}. Here they   are
  defined by
\begin{eqnarray*}
\BMOA(\rho_\Psi):=
\left\{f\in\mathcal H^2(\bB^n);\;
\sup_B\inf_{P\in \mathcal P_N(B)} \frac 1{\sigma(B)\rho_\Psi(\sigma(B))}\int_B|f-P|^2d\sigma<\infty \right\}.
\end{eqnarray*}
 where the integral is taken on the unit sphere, $f$ stands for the boundary values of the function and balls are defined for the Koranyi metric.  Moreover $\mathcal P_N(B)$ denotes the set of polynomials of degree $\leq N=N_\Psi$ in an appropriate basis, with $N$ large enough.
 
 When $\Psi=\Phi_1$, this space is usually referred as the space $\LMOA$ of
 functions of logarithmic mean oscillation. Duality  has been proven in $\bR^n$ by \cite{Janson}.
 Viviani proves it as a consequence of the atomic decomposition.
 In the context of holomorphic functions, this is also a consequence
 of the atomic decomposition and the continuity property of the
 Szeg\"o projection.

Our method allows us to characterize $\BMOA(\rho_\Phi)$ as the
class of symbols of Hankel operators which map $\H^\Phi$ into
$\H^1_{\mbox{\rm weak}}$.

\smallskip

As we said, we have chosen to allow the lower type of $\Phi$ to be arbitrarily small, and not only of upper type larger than $\frac{2n}{2n+1}$ (for the unit ball of $\bC^n$, or for a strictly pseudo-convex domain; for a general convex domain of finite type, the critical index is different). This induces many technical difficulties: for instance, it is not sufficient to deal with atoms with mean $0$ and we need extra moment conditions; in parallel, one has to deal with polynomials of positive degree to define  the dual space $\BMO$, and not only with constants.

 \smallskip
Here and in what follows, $\mathcal H(\bB^n)$ denotes the space of
holomorphic functions in $\bB^n$. For two functions $f$ and $g$, we use the notation $f\lsim g$ when
there is some constant $c$ such that $f(w)\le c\, g(w)$. Here
$w$ stands for the parameters that we are interested in
(typically, the constant $c$ will only depend on the geometry of
the domain under consideration).  We use the symbols $\gsim$ and
$\approx$ analogously.

\section{Statements of  results}\label{basic-facts}

\subsection{Growth functions and Orlicz spaces}

Let us give a precise definition for the {\sl growth functions}, which are used in the definition of Hardy-Orlicz spaces, see
also  \cite{Viviani}.
\begin{defn} Let $0<p\le 1$. A function $\Phi$ is called a growth function of order $p$ if it satisfies the following properties:
\begin{enumerate}
\item[(G1)] The function $ \Phi$ is  a homeomorphism of $[0,\infty)$ such that $\Phi(0)=0$.
Moreover, the function $\displaystyle t\mapsto \frac{ \Phi(t)}t$
is non-increasing.
\item [(G2)]  The function $\Phi$ is of lower type $p$, that is, there exists a constant
$c>0$ such that, for $s>0$ and $0<t\le 1$,
\begin{equation}\label{cond1}\Phi(st)\le ct^p\Phi(s).\end{equation}
\end{enumerate}
\end{defn}
We will also say that $\Phi$ is a {\sl growth function } whenever
it is a growth function of some order $p<1$. Two growth functions
$\Phi_1$ and $\Phi_2$ define the same Hardy-Orlicz spaces when
$\Phi_1\approx \Phi_2$. In particular, the growth function $\Phi$
of order $p$ is equivalent to the function $\dst\int_0^t
\frac{\Phi(s)}sds$, which is also a growth function of the same order and satisfies
the following additional property.
\begin{enumerate}
\item[(G3)] {\sl The function $ \Phi$ is concave. In particular,
$\Phi$ is sub-additive.}
\end{enumerate}
Our typical example $\Phi_p(t)=\frac{t^p}{\log(e+t)^p}$ satisfies
(G1) and (G2) for $p\leq 1$. The same is valid for the function
$\Phi_{p,\alpha}(t)=t^p(\log(C+t))^{\alpha p}$, provided that $C$
is large enough, for $p<1$ and any $\alpha$, or for $p=1$ and
$\alpha<0$. We modify it as above so that (G3) is also satisfied.
In the sequel we will assume that this modification has been
done, and use as well the notation $\Phi_p$ (or $\Phi_{p,\alpha}$)
 for the modified function.
 \begin{remark}\label{composition}
 When $\Phi$ and $\Psi$ are two growth functions, then
 $\Phi\circ \Psi$ is also a growth function.
 \end{remark}

 Remark also that $\Phi$ is doubling: more precisely,
\begin{equation}\label{doubling}
  \Phi(2t)\le 2\Phi(t),
\end{equation}
a property that will be largely used.

For $(X,d\mu)$ a measure space, we call $L^\Phi$ the corresponding
Orlicz space, that is, the space of functions $f$ such that
$$\|f\|_{L^{\Phi}}:= \int_{X}\Phi{(|f|)}d\mu<\infty.$$
The quantity $\|\cdot\|_{L^{\Phi}}$ is sub-additive, but is not
homogeneous. One may prefer the Luxembourg quasi-norm, which is
homogeneous but not sub-additive. It is defined as
$$\|f\|_{L^{\Phi}}^{\lux}=\inf\left\{\lambda>0:\;
 \int_{X}\Phi\left(\frac{|f(x)|}{\lambda}\right)d\mu(x)\le 1\right\}.$$
 It is easily seen that
 $$\|f\|_{L^{\Phi}}^{\lux}\lsim \min\{\|f\|_{L^{\Phi}},
 \|f\|_{L^{\Phi}}^p\},$$
 while
$$||f||_{L^{\Phi}}\lsim \max\{\|f\|_{L^{\Phi}}^{\lux},
 (\|f\|_{L^{\Phi}}^{\lux})^p\}.$$
 Endowed by the distance
 $||f-g||_{L^{\Phi}}$, $L^\Phi$ is a metric space.
 When $T$ is a linear continuous operator from $L^\Phi$ into the Banach
 space $ \mathcal{B}$, there exists some constant $C$ such that
 $$
\|Tf\|_{ \mathcal{B}} \leq C \|f\|_{L^{\Phi}}^{\lux}\leq C
\|f\|_{L^{\Phi}}.
$$
Conversely, a bounded operator is continuous.
 \subsection{Adapted geometry on the unit ball}
 Let us recall here the different geometric notions (see \cite{Rudin}) that will be necessary for the description of spaces of holomorphic functions.

 For $\zeta\in \bS^n$ and $w\in \overline {\bB^n}$, we note
$$d(\zeta,w):=|1-\langle \zeta,w\rangle|.$$
We recall that, when restricted to $\bS^n\times \bS^n$, this is a quasi-distance. For $\zeta_0\in \bS^n$ and $0<r<1$, we note $B(\zeta_0,r)$  the ball on $ \bS^n$ of center $\zeta_0$ and radius $r$ for the distance $d$. Recall that $\sigma(B(\zeta_0,r))\simeq r^{n}$. In particular,
\begin{equation}\label{doubling-balls}
    \sigma(B(\zeta_0,\lambda r))\simeq \lambda^n  \sigma(B(\zeta_0,\lambda r)),
\end{equation}
with constants that do not depend of $\zeta_0$ and $r$.

For each $\zeta_0\in\bS^n$, we choose an orthonormal basis $v^{(1)}, v^{(2)}, \cdots, v^{(n)} $ in $\bC^n$, such that $v^{(1)}$ is the outward normal vector to the unit sphere. In particular, we can choose the canonical basis for the point with coordinates $(1,0, \cdots, 0)$. Let us call $x_j+iy_j$ the coordinates of $z-\zeta_0$ in this basis. Then $y_1, y_2,\cdots, y_n,  x_2, \cdots, x_n$ can be used as coordinates for $\bS^n$ in a neighborhood of $\zeta_0$, say in the ball $B(\zeta_0, \delta)$. We can take $\delta$ uniformly for all points $\zeta_0$. We will speak of the {\sl special coordinates related to $\zeta_0$}.

Given $\zeta\in\bS^n$ we define the {\em admissible approach
region\/} $\A_\alpha(\zeta)$ as the subset of $\bB^n$ given by
$$
\A_\alpha(\zeta) =\{z=rw\in\bB^n:\,d(\zeta,w)=
|1-\langle\zeta,w\rangle| <\alpha(1- r)\}.
$$
We then define the {\em admissible maximal function} of the holomorphic function $f$ by
$\mathcal M_\alpha(f)$
\begin{equation}\label{f-star-gamma}
\mathcal M_\alpha(f)(\zeta) = \sup_{z\in\A_\alpha(\zeta)} |f(z)|.
\end{equation}

 \subsection{ Hardy-Orlicz spaces}

  Hardy-Orlicz spaces $\H^{\Phi}(\bB^n)$  have been defined in
  \eqref{hardy-orlicz}. We define on $\H^{\Phi}(\bB^n)$ the (quasi)-norms by
$$\Vert f\Vert_{\mathcal \H^{\Phi}(\bB^n)}:= \sup_{0<r<1}\int_{\bS^n}
{\Phi}(|f(rw)|) \,d\sigma(w),$$
$$\|f\|_{\H^{\Phi}(\bB^n)}^{\lux}=\inf\left\{\lambda>0:\;\sup_{0<r<1}
\int_{\bS^n}\Phi\left(\frac{|f(rw)|}{\lambda}\right)d\sigma(w)\le
1\right\},$$ which are finite for $f\in \H^{\Phi}(\bB^n)$ and
define  the same topology.

The assumptions on the growth function $\Phi$ give the inclusions

\begin{equation}\label{inclusion}
\H^{1}(\bB^n)\subset \H^{\Phi}(\bB^n)\subset \H^{p}(\bB^n).
\end{equation}

\medskip

A major property of Hardy spaces is given by the equivalence with
definitions in terms of maximal functions, which generalize in our
setting.

\begin{thm}\label{max-charact} Let $\alpha >0$. There exists a constant $C>0$ so that, for any $f\in\mathcal H^{\Phi}(\bB^n)$,

\begin{equation}\label{aire}\left\Vert \Phi(\mathcal M_\alpha(f))\right\Vert_{L^1(\bS^n)}\le C \Vert f\Vert_{\H^{\Phi}(\bB^n)}\end{equation}

\end{thm}
So the two quantities are equivalent.

Next we  define the real Hardy-Orlicz spaces on the unit sphere in
terms of {\sl atoms}.

For $\zeta_0\in\bS^n$, we consider an orthonormal basis $v^{(1)}, v^{(2)}, \cdots, v^{(n)} $ in $\bC^n$ giving rise to special coordinates related to $\zeta_0$. Recall that, if the coordinates for the basis $v^{(j)}$ are denoted by $w_j:=x_j+iy_j$, then $x_j$, for $j\geq 2$ and $y_j$ for $j\geq 1$ define local coordinates of $\bS^n$ inside the ball $B(\zeta_0, \delta)$. We call $\mathcal P_N(\zeta_0)$ the set of functions on $B(\zeta_0, \delta)$ which are polynomials of degree $\leq N$ in these $2n-1$ real coordinates. Remark that $\mathcal P_N(\zeta_0)$ does not depend on the choice of $v^{(2)}, \cdots, v^{(n)} $.

\begin{defn}\label{atom}
A square integrable function $a$ on $\bS^n$ is called an
atom of order $N\in\bN$  associated to the ball
$B:=B(\zeta_0,r_0)$,  for some $\zeta_0\in\bS^n$, if
the following properties are satisfied:
\begin{enumerate}
\item [{\rm(A1)}] $\text{supp\,}a\ss B $;
\item [{\rm(A2)}] $\int_{\bS^n} a(\zeta)P(\zeta)d\sigma(\zeta)=0$ for every $P\in \mathcal P_N(\zeta_0)$ when $r_0<\delta$.
\end{enumerate}
\end{defn}
The second condition is also called {\sl the moment condition}. It is only required for small balls.

We can now define the real Hardy-Orlicz spaces. Recall that the
term ``real" is related with the fact that the definition makes
sense for real functions, and does not require any assumption of
holomorphy. Here we consider complex valued functions, since in particular
we are interested in the fact that these spaces contain boundary
values (in the distribution sense) of holomorphic functions of
$\H^{\Phi}(\bB^n)$.
\begin{defn}\label{realHardy-Orlicz}
The real Hardy-Orlicz space $H^{\Phi}(\bS^n)$ is the space of
distributions $f$ on $\bS^n$ which can be written as the limit, in
the distribution sense, of  series
\begin{equation}\label{norm-convergence}
f=\sum_{j} a_j, \hspace{1.5cm}\sum_j \sigma(B_j){\Phi}(\Vert
a_j\Vert_2\sigma(B_j)^{-\frac 12})<\infty,
\end{equation}
where the $a_j$'s are atoms of order $N$,
associated to the balls $B_j$. Here $N$ is an integer chosen so that $N> N_p:=2n(\frac 1p-1)-1$.
\end{defn}
The (quasi) norm on $H^{\Phi}(\bS^n)$ is defined by
\begin{equation}\label{norm-atom}\|f\|_{H^{\Phi}}=\inf\left\{\sum_j \sigma(B_j){\Phi}(\Vert
a_j\Vert_2\sigma(B_j)^{-\frac 12}):\, f=\sum_j  a_j \right\}.\end{equation} It is also
sub-additive. In particular, with the distance between $f$ and $g$
given by $ \|f-g\|_{H^{\Phi}}$, $H^{\Phi}(\bS^n)$ is a complete
metric space. It is easily seen that the series in
(\ref{norm-convergence}) converges in {\em metric}.
Remark that convergence in $H^\Phi(\bS^n)$ implies convergence in
the sense of distribution.

We will see that the Szeg\"{o} kernel projects onto the space
$\H^{\Phi}(\bB^n)$.

\begin{remark} The condition on $N$ guarantees  that the Szeg\"{o} projection of the atom $a$ (or its maximal function) is well defined and has $L^\Phi$ norm uniformly bounded in terms of $\Phi(\|a\Vert_2\sigma(B)^{-\frac 12})\sigma (B)$. It follows from the theorems below that the space $H^{\Phi}(\bS^n)$ does not depend on $N> N_p$.
\end{remark}

Moreover, we have the following atomic
decomposition.

\begin{thm}\label{atomic-decom}
Let $N\in\bN$ be larger than $N_p$.  Given any
$f\in\H^{\Phi}(\bB^n)$ there exist atoms $a_j$ of order $N$ such
that $\sum_{j=0}^{\infty}a_j\in H^{\Phi}(\bS^n)$ and
$$
f=P_S\biggl(\sum_{j=0}^{\infty} a_j\biggr) =
\sum_{j=0}^{\infty}P_S(a_j).
$$
Moreover
$$
\sum_{j=0}^{\infty}\sigma(B_j){\Phi}(\Vert
a_j\Vert_2\sigma(B_j)^{-\frac 12})\approx\|f\|_{\H^{\Phi}(\bB^n)}.
$$
\end{thm}
As in the atomic decomposition of  Hardy spaces of
$\bR^n$, the order of moment conditions of the atoms can be chosen
arbitrarily large. Having optimal values has no importance later
on, which allows to adapt easily proofs to a class of domains including convex domains of
finite type and strictly pseudo-convex domains, for which the optimal values of $N_p$ are different. The fact that atoms may satisfy
moment conditions up to an arbitrary large order will play a crucial role
for the factorization.

Szeg\"{o} projections of atoms are best described in terms of
molecules, which we introduce now.
\begin{defn}\label{molecule}
A holomorphic function $A\in \H^2(\bB^n)$ is called a molecule of
order $L$, associated to the ball $B:=B(z_0,r_0)\subset \bS^n$, if
it satisfies
\begin{equation}\label{mol}
\Vert A\Vert_{{\rm
mol}(B,L)}:=\left(\sup_{r<1}\int_{\bS^n}\left(1\,+\,\frac{d(z_0,\xi)^{L+n}}{r_0^{L+n}}\right)|A(r\xi)|^2
\frac{d\sigma(\xi)}{\sigma(B)}\right)^{1/2}<\infty.
\end{equation}
\end{defn}
\begin{prop}\label{molecule-property}
For an atom $a$ of order $N$ associated to  the ball
$B\subset \bS^n$, its Szeg\"o projection
 $P_S(a)$ is a molecule associated to $B$ of any order $L<N+1$.
It satisfies
$$\Vert A\Vert_{{\rm
mol}(B,L)}\lsim \Vert a\Vert_2\sigma(B)^{-\frac 12}.$$
\end{prop}
\begin{prop}\label{at-gives-mol}
Any molecule $A$ of order $L$ so that
$L>L_p:=2n(1/p-1)$ belongs to $\mathcal H^\Phi(\bB^n)$ with
$$\Vert A\Vert_{\mathcal H^\Phi}\lsim \Phi(\Vert A\Vert_{{\rm
mol}(B,L)})\sigma(B).$$
\end{prop}

The atomic decomposition and  the previous propositions have as
corollaries the
 molecular decomposition of functions in $\mathcal H^\Phi(\bB^n)$, the continuity
 of the Szeg\"o projection, and the identification of the dual
 space. Let us begin with molecular decomposition.
\begin{thm}\label{mol-dec}
For any $f\in\H^{\Phi}(\bB^n)$, there exists molecules $A_j$ of
order $L$, $L>L_p$, associated to the balls $B_j$, so that $f$
may be written as
$$f=\sum_{j} A_j$$
with $\Vert f\Vert_{\H^{\Phi}(\bB^n)}\approx \sum_j  \Phi(\Vert
A_j\Vert_{{\rm mol}(B_j,L)})\sigma(B_j).$
\end{thm}
The continuity of the Szeg\"o projection is also a direct
consequence of the atomic decomposition and the fact that an atom
is projected into a molecule.
\begin{thm}\label{boundedness}
The Szeg\"o projection extends into a continuous operator,
$$
P_S: H^{\Phi}(\bS^n)\rightarrow\H^{\Phi}(\bB^n).
$$
\end{thm}
Before giving the duality statement, let us first define the
generalized $\BMO(\varrho)$-spaces as follows. We assume that $\varrho$ is a continuous increasing function from $[0,1]$ onto  $[0,1]$, which is of upper type $\alpha$, that is,
\begin{equation}\label{upper}
    \varrho(st)\leq s^\alpha \varrho(t)
\end{equation}
 for $s>1$, with $st\leq 1$. We then define
$$\BMO(\varrho)=\left\{f\in L^2(\bS^n)\;;\;\sup_B \inf_{P\in \mathcal P_N(B)}\frac 1{\varrho(\sigma(B))\sigma(B)}\int_B|f-P|^2d\sigma<\infty \right\}.$$
Here, for $B$ a ball of center $\zeta_B$, assumed to be of radius $r<\delta$, we note $\mathcal P_N(B):=\mathcal P_N(\zeta_B)$.  The integer $N$ is taken  large enough, say $N>2n\alpha-1$. Before going on, let us make some remarks.
\begin{remark} Instead of the infimum on $P\in \mathcal P_N(B)$, we can take the function $E_N f$, with $E_N$ the orthogonal projection (in $L^2(B)$) onto $\mathcal P_N(B)$.
\end{remark}
\begin{remark} The definition does not depend on $N>2n\alpha -1$. We will not prove this and refer to \cite{BPS2} for a proof for $\alpha<1/2$. It is a consequence of duality and atomic decomposition. 
\end{remark}

\begin{remark} One may prove that, as in the Euclidean case (see \cite{Janson}) when $\varrho$ is of upper type less than $1/2n$ and satisfies the Dini condition
$$\int_r^1 \frac{\varrho(s)}{s^2} ds \lsim \varrho(r),$$
then $\BMO(\varrho)$ coincides with the Lipschitz space $\Lambda(\varrho)$, defined as the space of bounded functions such that
$$|f(z)-f(\zeta)|\leq \varrho(d(z,\zeta)^n).$$
\end{remark}

Spaces $\BMO(\varrho)$ have been introduced by Janson \cite{Janson} in $\bR^n$, and proved to be the dual spaces of maximal Hardy-Orlicz 
spaces related to the growth function $\Phi$ when  $\dst
\varrho(t)=\varrho_\Phi(t):=\frac 1{t\Phi^{-1}(1/t)}$. With our
definition of $H^{\Phi}(\bS^n)$ in terms of atoms, this duality is
straightforward, as remarked by Viviani ( \cite{Viviani}). For
holomorphic Hardy-Orlicz spaces, we have also
\begin{thm}\label{duality}
 The dual space of $\H^\Phi(\bB^n)$ is the space $\BMOA(\varrho)$, defined by
$$\BMOA(\varrho)=\left\{f\in\H^2(\bB^n);\sup_B\inf_{P\in \mathcal P_N(B) }\frac 1{\varrho(\sigma(B))\sigma(B)}\int_B|f-P|^2d\sigma<\infty\right\}$$
where $\dst \varrho(t)=\varrho_\Phi(t):=\frac 1{t\Phi^{-1}(1/t)}$.
The duality is given by the limit as $r<1$ tends to $1$ of scalar
products on spheres of radius $r$.
\end{thm}
In other terms, $\BMOA(\varrho)$ is the space of holomorphic
functions of the Hardy space $\H^2(\bB^n)$ whose boundary values
belong to $\BMO(\varrho)$.
\subsection{Products of functions and Hankel operators}
We now have all prerequisites to study the
product of a function $h\in\H^\Phi(\bB^n)$ with a function in
$b\in\BMOA(\bB^n)$. Remark that, using \eqref{inclusion}, we
already know that the product is well defined as the product of a
function of $\H^p(\bB^n)$ and a function of $\H^s(\bB^n)$ for all
$1<s<\infty$. So it is a function of $\H^q(\bB^n)$ for $q<p$. We
want to replace this first inclusion by a sharp statement.
\begin{prop}\label{prod}
The product maps continuously $\H^\Phi(\bB^n)\times  \BMOA(\bB^n)$
 into $\mathcal H^{\Psi}(\bB^n)$, where $\dst
\Psi(t)=\Phi\left(\frac t{\log(e+t)}\right).$
\end{prop}
\begin{proof}
We know that $\Psi$ is also a growth function by Remark
\ref{composition}. We prove more: using John Nirenberg
Inequality, we know that a function $b$ in $\BMO$ is also in the exponential class. More
precisely, we only use the fact that $b(r\cdot)$ is uniformly
in the exponential class, and prove that, for such a function $b$ and for a function $h\in\H^\Phi(\bB^n)$,
the product $b\times h$
 is continuously embedded in $\mathcal H^{\Psi}(\bB^n)$. We start from the
following elementary inequality, see \cite{BIJZ}. For any $u,v>0$,
$$\frac{uv}{\log(e+uv)}\le u+e^v-1.$$
It follows that
$$\Psi(uv)\lsim \Phi(u+ e^v-1)\lsim \Phi(u)+e^v-1.$$ When
$u$ and $v$ are replaced by measurable positive functions on the
measure space $(X,d\mu)$, we have, by homogeneity of the
Luxembourg norms, the inequality
$$\Vert fg\Vert_{L^{\Psi}}^{\lux}\lsim \Vert
f\Vert_{L^{\Phi}}^{\lux}\Vert g\Vert_{\exp L}^{\lux}.$$
We refer to \cite{VT} for more general H\"older inequality on Orlicz spaces.

Let us come back to Hardy spaces. Applying this inequality on each
sphere of radius less than $1$, we conclude that
\begin{equation}\label{ineq-product}
\Vert fg\Vert_{\H^{\Psi}}^{\lux} \lsim \Vert
f\Vert_{\H^{\Phi}}^{\lux}\Vert g\Vert_{\exp L}^{\lux} \lsim \Vert
f\Vert_{\H^{\Phi}}^{\lux}\Vert g\Vert_{\BMOA} .
\end{equation}
\end{proof}
We  are going to prove  converse statements.
\begin{thm}\label{factorization}
 Let $A$ be a molecule associated to the ball $B$. Then $A$ may be written as $fg$,
 where $f$ is a molecule and $g$ is in $\BMOA (\bB^n)$. Moreover, $f$ and $g$ may be chosen such
 that
 $$\|g\|_{\BMOA (\bB^n)}\lsim 1,\hspace{1.5cm} \Vert
f\Vert_{{\rm mol}(B,L')}\lsim \frac{\Vert A\Vert_{{\rm
mol}(B,L)}}{\log(e+\sigma(B)^{-1})}$$ when $L'<L$.  In particular,
if $\Psi(\Vert A\Vert_{{\rm mol}(B,L)})\sigma(B)\leq 1$, then
$$\Phi(\Vert f\Vert_{{\rm mol}(B,L')})\lsim \Psi(\Vert
A\Vert_{{\rm mol}(B,L)}).$$\end{thm}
\begin{thm}\label{fact-dec}
 Given any
$f\in\H^{\Psi}(\bB^n)$ there exist $f_j\in\H^\Phi(\bB^n)$, $g_j\in
\BMOA(\bB^n)$, $j\in\bN$, with the norm of $g_j$ bounded by $1$,
such that
$$
f=\sum_{j=0}^{\infty} f_j g_j.
$$
Moreover, we can take for $f_j$ a molecule and, for $\Vert f\Vert_{\H^\Psi}\leq 1$, we have the
equivalence $$\Vert f\Vert_{\H^\Psi}\approx \sum_j  \Phi(\Vert
f_j\Vert_{{\rm mol}(B_j,L)})\sigma(B_j).$$
 In particular,
$$\sum_{j=0}^\infty \Vert f_j\Vert_{\H^\Phi}\Vert
g_j\Vert_{\BMOA}\lsim \Vert f\Vert_{\H^\Psi}.$$
\end{thm}

 As
a corollary, we obtain the following characterization of bounded
Hankel operators. Recall that, for $b\in\mathcal H^2(\bB^n)$, the
 (small) Hankel operator $h_b$ of symbol $b$ is given, for  functions $f\in
H^2(\bB^n)$, by $h_b(f)=P_S(b\overline f)$.

\begin{cor}\label{Hankel}
Any Hankel operator $h_b$  extends into a continuous operator from
$\mathcal H^\Phi(\bB^n)$ to $\mathcal H^1(\bB^n)$ if and only if
$b\in (\mathcal H^{\Psi}(\bB^n))'=\BMOA(\varrho_\Psi)$.
\end{cor}
The proof is elementary once we know the previous statements. We give it here. \begin{proof} Let $h_b$ be a Hankel operator of
symbol $b$. Let us first assume that $b$ belongs to
$\BMOA(\varrho_\Psi)$. Then, for any $g$ in $\BMOA$,  we have
\begin{eqnarray*}
|\langle h_b(f),g\rangle|&=&|\langle P_S(b\overline{f}),g\rangle|=|\langle b,fg\rangle|\\
&\lsim &\Vert b\Vert_{\BMOA(\varrho_\Psi)}\Vert fg\Vert_{\H^\Psi}^{\lux}
\lsim \Vert b\Vert_{\BMOA(\varrho_\Psi)}\Vert f\Vert_{\H^\Phi}^{\lux}
\Vert g\Vert_{\BMOA}.
\end{eqnarray*} It follows that $h_b$ is bounded from $\H^\Phi(\bB^n)$ to
$\H^1(\bB^n)$, which we wanted to prove.

Conversely, assume now that $h_b$ is bounded from $\H^\Phi(\bB^n)$
to $\H^1(\bB^n)$ and prove that $b$ belongs to the dual of
$\H^{\Psi}(\bB^n)$. It is sufficient to prove that there exists
some constant $C$ such that
$$|\langle b,f\rangle|\leq C $$
when $f$ belongs to a dense subset of functions in
$\H^{\Psi}(\bB^n)$, with $\Vert f\Vert_{\H^\Psi}\lsim 1$. Because
of Theorem \ref{fact-dec}, it is sufficient to test on such
functions $f$, which may be written as a finite sum of products
$f_j g_j$. More precisely,
\begin{eqnarray*}
|\langle b,f\rangle|&=&|\langle b,\sum f_jg_j|\le \sum_j|\langle P_S(bf_j),g_j\rangle|\\
&=&\sum_j|\langle h_b(f_j),g_j\rangle|\le |||h_b|||\sum_j\Vert
f_j\Vert_{\H^\Phi}^{\lux}\Vert g_j\Vert_{BMOA}\leq C.
\end{eqnarray*}
It ends the proof.
\end{proof}

 All these results may be extended to the more
general setting of strictly pseudoconvex domains or of convex
domains of finite type in $\bC^n$. We give a sketch of the proofs
in section \ref{Casgeneral}.

\section{Maximal characterizations of Hardy-Orlicz spaces}

Let us prove the equivalent characterization of $\H^\Phi$ spaces,
given in Theorem \ref{max-charact}. In order to adapt the proofs
given for usual Hardy spaces, we need the following lemma. Here
$\mathcal M^{HL}$ denotes the Hardy-Littlewood maximal operator
related to the distance on the unit sphere. In fact, the statement
is valid in the general context of spaces of homogeneous type. In
particular we will also use it for the maximal operator on the sphere
related to the Euclidean distance.
\begin{lem}\label{HL}
Let $\Phi$ be a growth function of order $p$ and $\beta<p$. There exists a constant $C>0$ so that, for any
measurable function $f$,
$$\int_{\bS^n}\Phi\left(\mathcal M^{HL}(|f|^{\beta})^{\frac{1}{\beta}}\right)d\sigma \le C \int_{\bS^n}\Phi(|f|)d\sigma.$$
\end{lem}

\begin{proof}
Let us note $g:=|f|^{\beta}$. We only use the fact that
$$ t \sigma\left(\mathcal M^{HL}(g)\ge t\right)\lsim \int_{\{g\ge t/2\}}g d\sigma, $$ which is a  consequence of the weak (1,1) boundedness of
$\mathcal M^{HL}$.

Denote by $\Psi$ the function defined by $\Psi(t):= \Phi(t^{\frac
1 \beta} )$, which is of lower type $p/\beta >1$. In particular,
\begin{equation}\label{lower}
\int_0^s\frac{\Psi(t)}{t^{2}}dt=s^{-1}\int_0^1\frac{\Psi(st)}{t^2}dt\lsim
\frac{\Psi(s)}{s}
\end{equation}
since $\int_0^1t^{p/\beta -2}dt$ is finite. It follows, cutting
the integral into intervals $(2^k,2^{k+1})$, that
\begin{equation}\label{lower-bis}\sum_{k ;\;
s>2^k}2^{-k}\Psi(2^k)\lsim \frac {\Psi(s)}s.\end{equation}
 Now, we
have to estimate
\begin{eqnarray*}
\int_{\bS^n}\Psi(\mathcal M^{HL}(g))d\sigma & \leq&\sum_k \Psi(2^k)\sigma\left(\mathcal M^{HL}(g)\ge 2^{k-1}\right)\\
&\lsim & \sum_k 2^{-k}\Psi(2^k)\int_{\{g\ge 2^ {k-2}\}} g d\sigma .
\end{eqnarray*}

Exchanging the integral and the sum and using \eqref{lower-bis},
we obtain that the left hand side is bounded by $C \int_{\bS^n}
\Psi(g)d\sigma $, which we wanted to prove.
\end{proof}
\begin{proof}[Proof of Theorem \ref{max-charact}]
We proceed in two steps, as it is classical. Let us note
$$\mathcal{M}_0(f)(\zeta) = \sup_{0<r<1}|f(r\zeta)|$$
the radial maximal function. We first prove that
\begin{equation}\label{sup}\left\Vert \Phi(\mathcal{M}_0(f))\right\Vert_{L^1(\bS^n)}\le C \Vert f\Vert_{\H^{\Phi}(\bB^n)}.\end{equation}
Let  $\beta<p$, $\Psi$  and $g= |f|^\beta$ be as before. The
function $g$ is sub-harmonic, and satisfies the condition
$$\sup_{0<r<1}\int_{\bS^n}\Psi(g(r\zeta))d\sigma(\zeta)<\infty.$$
We claim that there exists some constant $C$, independent of $g$,
such that
\begin{equation}\label{sup-bis}
  \int_{\bS^n}\Psi(\sup_{0<r<1}g(r\zeta))d\sigma(\zeta)\le C\sup_{0<r<1}\int_{\bS^n}\Psi(g(r\zeta)d\sigma(\zeta),
\end{equation}
which will immediately imply \eqref{sup}. The proof of
\eqref{sup-bis} follows the same lines as in the unit disc. Assume
first that $g$ extends into a continuous function on the closed ball and
call $\tilde g$ the function on the unit sphere that coincides
with this extension. With this assumption,  the right hand side is the integral of $\Psi (\tilde g)$. Then it follows from the maximum principle
that $g\leq G$, where $G$ is the Poisson integral of $\tilde g$.
Moreover, we know that $\sup_{0<r<1}g(r\zeta)$ is bounded by the
Hardy Littlewood maximal function (for the Euclidean metrics on
the unit sphere) of $\tilde g$. We conclude for the inequality
\eqref{sup-bis} by using the previous lemma, or its proof, in the
context of this maximal function. To conclude for general $g$, it
is sufficient to see that Inequality \eqref{sup} is valid for $g$
once it is valid for all $g(r\cdot)$, with $0<r<1$.

 Let $\tilde f$ be the a. e. boundary values of $f$, which we know to exist since $f$ belongs to $\mathcal
 H^{p}(\bB^n)$ by \eqref{inclusion}. Remark that once we have done this first step, we also
know, using Fatou's lemma, that $\Vert \Phi(|\tilde
f|)\Vert_{L^1(\bS^n)}\le \Vert f\Vert_{\mathcal H^{\Phi}}$.

\medskip

Next, we recall that (see for instance \cite{G}, or \cite{St2} for
the Euclidean case) that we have
the inequality
\begin{equation}\label{max-in}
  \mathcal M_\alpha(f)^\beta\le C_\alpha \mathcal
M^{HL}\left(\mathcal {M}_0(f)^{\beta}\right).
\end{equation}
 We then use Lemma
\ref{HL} to conclude for the proof of Therorem \ref{max-charact}.

\end{proof}

We need stronger characterizations of $\H^{\Phi}(\bB^n)$ for
the atomic decomposition. First, remark that when looking at the
proof of \eqref{max-in}, one observes that the constant $C_\alpha$
has a polynomial behavior when $\alpha$ tends to $\infty$. In the
Euclidean case, details are given in \cite{St2}. This means in
particular, using the fact that $\Phi$ is doubling, that for some
large $N_0$ and all $\alpha>0$, we have the inequality
\begin{equation}\label{aire-bis}\left\Vert \Phi(\mathcal M_\alpha(f))\right\Vert_{L^1(\bS^n)}\le C(1+\alpha)^{N_0} \Vert f\Vert_{\H^{\Phi}(\bB^n)}.\end{equation}

Let us consider now the tangential variant of admissible maximal
operators, defined by
\begin{equation}\label{f-double-star-M}
\mathcal N_M(f) (\zeta) = \sup_{rw\in\bB^n} \biggl( \frac{1- r}
{(1-r)+d(\zeta,w)}\biggr)^M |f(rw)|.
\end{equation}
Here $d(\zeta,w)$ denotes the pseudo-distance on $\bS^n$, given as before by
$d(\zeta,w):=|1-\langle \zeta,w\rangle|$. We claim that the
following identity holds.
\begin{equation}\left\Vert \Phi(\mathcal N_M(f))\right\Vert_{L^1(\bS^n)}\le C \Vert f\Vert_{\H^{\Phi}(\bB^n)}.\end{equation}
 Using the definition, we have
\begin{eqnarray*}
\mathcal N_Mf(\zeta)&=& \sup_{k\in\bN}\sup_{rw\in\mathcal
A_{2^k}(\zeta)}\left(\frac{1-r}{(1-r)+d(\zeta,w)}\right)^M|f(rw)|\cr
&\lsim & \sup_{k\in\bN}2^{-kM}\mathcal M_{2^k}f(\zeta).
\end{eqnarray*} It then follows that
\begin{eqnarray*}
\Vert  \Phi(\mathcal N_M(f))\Vert_{L^1(\bS^n)}\le
\sum_{k\in\bN}\Vert \Phi(2^{-kM}\mathcal
M_{2^k}f)\Vert_{L^1(\bS^n)}&\le&
\sum_{k\in\bN}2^{-kMp}\Vert\Phi(\mathcal
M_{2^k}f)\Vert_{L^1(\bS^n)}.\end{eqnarray*} For $Mp>N_0$ we can
conclude after having  used \eqref{aire-bis}.

Le us now introduce the {\sl grand maximal function}. Firstly, we define the set of {\em smooth bump
 functions  at }
$\zeta$, which we note
$\K_\alpha^N(\zeta)$, as the set of smooth functions $\varphi$ supported in $B(\zeta_0,r_0)$ for some $\zeta_0 \in\A_\alpha(\zeta)$ 
and normalized in the following way. In the neighborhood of $\zeta_0$, when we use special coordinates related to $\zeta_0$, the unit sphere coincides with the graph $\Re w_1=h(\Im w_1, w')$, with $w'=(w_2, \cdots, w_n)$ and $h$ a smooth function. We note $w_j=x_j+y_j$, and consider all derivatives $D^{(k,l)}\varphi$, where $D^{(k,l)}$ consists in $k$ derivatives in $x'$ or $y'$, and $l$ derivatives in $y_1$. We assume that bump functions $\varphi \in \K_\alpha^N(\zeta)$ satisfy the inequality
$$\sum_{k+l\le N, }
\|D^{(k,l)}\varphi \|_{L^\infty(B(\zeta_0,r_0))}r_0^{k/2+l}\le \sigma(B)^{-1}.
$$

The {\em grand  maximal function\/} is defined as
\begin{equation}\label{K-gamma-M-zeta-function}
K_{\alpha,N} (f)(\zeta)= \sup_{ \varphi\in\K_\alpha^N(\zeta)}
\big| \lim_{r\rightarrow 1}\int_{S^n} f(r\zeta)\varphi(\zeta)
d\sigma(\zeta)|.
\end{equation}
The fact that the limit exists for $f\in \H^\Phi(\bB^n)\subset
\H^p(\bB^n)$ is due to the fact that holomorphic functions in
Hardy spaces have boundary values as distributions.

We use the following inequality (see \cite{GP}, and \cite{St2}
for the Euclidean case).
\begin{lem}\label{one}
With the definitions above, there exist $c=c(\bB^n)$ and $\tilde
N=\tilde N(\alpha,N)$ such that$$
K_{\alpha,N}f(\zeta) \lsim
\mathcal M_{c\alpha}(f) (\zeta) + \mathcal N_{\tilde N}(f) (\zeta).
$$
\end{lem}

We now turn to the atomic decomposition. We first prove in the next section that holomorphic extensions of functions in $H^{\Phi}(\bS^n)$ are functions of the Hardy-Orlicz space.
\section{Atoms and molecules \ref{boundedness}}

We first consider the Szeg\"o projection of atoms and prove the
following lemma.
\begin{lem}\label{estimA}

Let $a$ be an atom of order $N$
associated to the ball $B=B(\zeta_0,r_0)$, and let $A=P_S(a)$.
Then $A$ satisfies the following estimates.
\begin{equation}\label{close}
\sup_{0<r<1}\int_{B(\zeta_0, 2r_0)}\Phi(|A(rw)|)\frac{d\sigma(w)}{\sigma(B)}\lsim
\Phi(\Vert a\Vert_2\sigma(B)^{-\frac 12}),
\end{equation}
\begin{equation}\label{far}
|A(r\zeta)| \lsim \biggl(\frac{r_0}{d(\zeta,\zeta_0)}\biggr)^{
n+\frac{N+1}2} \Vert a\Vert_2\sigma(B)^{-\frac 12} \  \ \mbox{\rm for} \
d(\zeta,\zeta_0)\ge 2r_0.
\end{equation}
\end{lem}
\bpf Let us prove \eqref{close}. We assumed that $\Phi$ is
concave. In particular, if $d\mu$ is a probability measure and $f$
a positive function on the measure space $(X, d\mu)$, then we have
Jensen Inequality
\begin{eqnarray}\label{jensen}
  \int_X\Phi(f)d\mu\leq \Phi\left(\int_X f
d\mu\right)\leq \Phi \left(\|f\|_{L^2(X, d\mu)}\right).
\end{eqnarray}
 If we use
it for the  measure $d\sigma$ on $B(z_0, 2r_0)$ after
normalization, we find that

\begin{equation}\label{jensen2}
\sup_{0<r<1}\int_{B(\zeta_0,
2r_0)}\Phi(|A(rw)|)\frac{d\sigma(w)}{\sigma(B)}\lsim\Phi\left(
\frac{\|A\|_{\H^{2}}}{(\sigma(B))^{1/2}}\right).
\end{equation}
Since the Szeg\"o projection is bounded in $L^2$,  we have the
inequality $$\|A\|_{\H^{2}}\leq \|a\|_{L^{2}}$$ and conclude for
\eqref{close}.

\bigskip

The inequality \eqref{far} is classical and used
for classical Hardy spaces. It is a consequence of the
estimates of the Sz\"ego kernel, which are explicit for the unit ball. Without loss of generality we can assume that $\zeta_0=(1, 0, \cdots, 0)$, so that the coordinates related to $\zeta_0$ may be taken as the ordinary ones. Otherwise we use the action of the unitary group. In the neighborhood of $\zeta_0$, the unit sphere coincides with the graph $\Re w_1=h(\Im w_1, w')$, with $w'=(w_2, \cdots, w_n)$. We recall that $S(\zeta, w)=c_n (1-\zeta.\bar w)^{-n}$. In the following estimates, we are interested in estimates on $D^{(k,l)}_wS(r\zeta, (h(t_1, s'+it')+it_1, w'))$, where $D^{(k,l)}$ consists in $k$ derivatives in $s'$ or $t'$, and $l$ derivatives in $t_1$. It follows from elementary computations  that
$$|D^{(k,l)}_wS(r\zeta, (h(t_1, s'+it')+it_1, w'))|\leq C  (|\zeta'|^k+|w'|^k)|1-\zeta.\bar w|^{-(n+k+l)}.$$
 In particular, for $d(w, \zeta_0)<r_0$ and $\zeta \notin B(\zeta_0, 2r_0)$, we know that
 $|1-\zeta.\overline w|\simeq |1-\zeta.\overline {\zeta_0}|\gsim r_0$. In particular, we have $|w'|\lsim |1-\zeta.\overline {\zeta_0}|^{\frac 12}$, and the same for $|\zeta'|$. So, the following holds
 \begin{equation}\label{estimateSzego}|D^{(k,l)}_wS(r\zeta, (h(t_1, s'+it')+it_1, w'))|\leq C |1-\zeta.\overline{ \zeta_0}|^{-(n+\frac k2+l)} \end{equation}
We use  the vanishing moment condition, in the computation of
$$P_S a(r\zeta)=\int S(r\zeta, w)a(w)d\sigma(w),$$
to replace $S(r\zeta,\cdot)$ by $S(r\zeta,\cdot)-P$, where $P$ is its Taylor polynomial at order $N$. By Taylor's formula,  the rest may be bounded by the sum, for $k+l=N+1$, of the quantities $|t_1|^l|w'|^k |1-z.\overline{ \zeta_0}|^{-(n+\frac k2+l)}$.
Using the fact that $|t_1|^l|w'|^k \lsim r_0^{\frac k2 +l}$, we have
$$|S(r\zeta, w)-P(w)|\leq C\frac {r_0^{\frac {N+1}2}}{d(z, \zeta_0)^{n+\frac {N+1}2}}.$$
This gives the result, using the fact that $\sigma(B)\lsim r_0^n$.
\epf
\bpf[Proof of Proposition \ref{molecule-property}] The fact  that
 $P_Sa$ is a molecule is classical. We give the proof for completeness. Coming back to  the definition of
$\Vert P_S(a)\Vert_{\text{mol}(B,L)}^2$ given in \eqref{mol}, we cut the
integral involved into two pieces. We already know that the
integral on $B(\zeta_0, 2r_0)$ satisfies the right estimate. So is sufficient to show that
$$
\int_{\bS^n\setminus B(\zeta_0,
2r_0)}\left(\frac {d(\xi,
\zeta_0)}{r_0}\right)^{L+n}\sup_r|P_Sa(r\xi)|^2\frac{d\sigma(\xi)}{\sigma(B)}\lsim
\|a\|_2^2.$$ Using \eqref{far}, it is a consequence of the
estimate
\begin{equation}\label{rest}
  \int_{\bS^n\setminus B(\zeta_0,
2r_0)}\left(\frac {r_0}{d(\xi,
\zeta_0)}\right)^M\frac{d\sigma(\xi)}{\sigma(B)}\leq C
\end{equation}
for some constant $C$ that does not depend on $\zeta_0$ and $r_0$,
when $M>n$ (see for instance \cite{Rudin}).

\epf

 \bpf[Proof of Proposition \ref{at-gives-mol}] Let
$A$ be a molecule of order $L$ associated to $B:=B(z, r)$. We want
to prove that $A$ belongs to $\mathcal H^\Phi(\bB^n)$ for $L$
large enough, with the estimate
$$\Vert A\Vert_{\mathcal H^\Phi}\lsim \Phi(\Vert A\Vert_{\text{mol}(B,L)})\sigma(B).$$
Let us note $B_k:=B(z, 2^k r)$. It is sufficient to prove that,
for $g$ a positive function on the unit sphere,
$$\int_{\bS^n}\Phi(g) \frac{d\sigma}{\sigma(B)}\lsim \Phi
\left(\left(\int_{\bS^n}\left(\frac {d(z,\xi)}r\right )^{L+n}
g(\xi)^2\frac{d\sigma(\xi)}{\sigma(B)}\right)^{1/2}\right).$$
Cutting the integral into pieces, it is sufficient to prove that
$$\int_{B}\Phi(g) \frac{d\sigma}{\sigma(B)}\lsim \Phi
\left(\left(\int_{B}
g(\xi)^2\frac{d\sigma(\xi)}{\sigma(B)}\right)^{1/2}\right),$$
which is a direct consequence of Jensen Inequality \eqref{jensen}
as before, and, for $k\ge 1$,
$$\int_{B_k\setminus B_{k-1}}\Phi(g) \frac{d\sigma}{\sigma(B)}\lsim 2^{-k\varepsilon}\Phi
\left(\left(2^{k(L+n)}\int_{B_k\setminus B_{k-1}}
g(\xi)^2\frac{d\sigma(\xi)}{\sigma(B)}\right)^{1/2}\right)$$ for
some $\varepsilon
>0$. To prove this last inequality, we use again Jensen Inequality
\eqref{jensen} for the measure $d\sigma$ on $B_k\setminus
B_{k-1}$, divided by its total mass $\sigma(B_k\setminus
B_{k-1})\approx 2^{kn}\sigma(B)$. This gives
$$\int_{B_k\setminus B_{k-1}}\Phi(g) \frac{d\sigma}{\sigma(B)}\lsim 2^{kn}\Phi
\left(\left(2^{-kn}\int_{B_k\setminus B_{k-1}}
g(\xi)^2\frac{d\sigma(\xi)}{\sigma(B)}\right)^{1/2}\right).$$ We
conclude by using the fact that $\Phi$ is of lower type $p$, which
allows to write that $2^{kn}\Phi(t)\lsim \Phi(2^{kn/p}t)$. It is
sufficient to choose $L>n(2/p-2)$. \epf

\section{Proof of the atomic decomposition Theorem \ref{atomic-decom}}

Let $f$ be a fixed function in $\mathcal H^{\Phi}$. As noticed
before,  $f$ admits boundary values defined
a.e. on $\bS^n$, that we still denote by $f$.

 We fix also $N$ an integer larger than $N_p$.

Let $k_0$ be the least integer such that
\begin{equation}\label{k-0}
\| \Phi(K_{\alpha,M}(f)+\mathcal M_\alpha(f))\|_{L^1(\bS^n)} \le 2^{k_0}.
\end{equation}
For a positive integer $k$, we define
\begin{equation}\label{O-k}
\mathcal{O}_k =\{ z\in\bS^n:\, K_{\alpha,M}f(z)+\mathcal M_\alpha(f)(z)>2^{k_0+k} \}.
\end{equation}
For each $k$, we then fix a Whitney covering  $\{B^k_i\}$ of $\mathcal O_k$.
As it is usual, one can associate to $f$ an atomic decomposition (see \cite{GL} for a proof for Hardy spaces in the unit ball, we refer also  to \cite{GP} for a proof in the general context considered in the last paragraph).

Namely, there exist a function $h_0$ and atoms $b_i^k$ corresponding to the Whitney covering $\{B_i^k\}$ so that the following equality holds in the distribution sense and almost everywhere.

\begin{equation}\label{atomic-decom-equation-1}
f=h_0 +\sum_{k=0}^{\infty}\sum_{i=0}^{\infty} b_i^k.
\end{equation}
 Here, $h_0$ is a  so called "junk atom" bounded by $c2^{k_0}$ while   the $b_i^k$'s are atoms supported in the $B_i^k$'s, bounded by $c2^{k+k_0}$, with moment conditions of order $N$.

Since $\|b_i^k\|_2 \sigma(B_i^k)^{-\frac 12}\leq \|b_i^k\|_\infty$, it is sufficient   to prove that
$$\sum_{i,k}\sigma(B_i^k){\Phi}(\Vert b_i^k\Vert_\infty)<\infty.$$ We have
\begin{align*}
\sum_{k=0}^{\infty}\sum_{i=0}^{\infty} \sigma(B_i^k){\Phi}(\Vert b_i^k\Vert_\infty)&\le \sum_{k=0}^{\infty}{\Phi}(2^{k+k_0})\sigma(\O_k)\\
&\le  c \int_1^\infty\frac{\Phi(t)}t\sigma\bigl( \{ \zeta\in\bS^n:\,
    K_\alpha^M f(\zeta)+\mathcal M_\alpha(f)(\zeta) \ge t\}\bigr) dt\\
&\lsim  c \int_1^\infty \Phi'(t)\sigma\bigl( \{ \zeta\in\bS^n:\,
    K_\alpha^M f(\zeta)+\mathcal M_\alpha(f)(\zeta) \ge t\}\bigr) dt\\
    &\le \Vert \Phi(K_\alpha^M f)\Vert_{L^1(\bS^n)}+\Vert\Phi(\mathcal M_\alpha (f)) \Vert_{L^1(\bS^n)}\\
& \le c \|f\|_{\H^{\Phi}}.
\end{align*}

\epf

As we
pointed out before, the atomic decomposition allows to obtain a
lot of result such as the molecular decomposition that we are
going to consider now.

\section{Factorization Theorem and Hankel operators}\label{Sectionfactorization}

Let us prove now the factorization theorem \ref{factorization}.
Let $A$ be a molecule associated to the ball $B=B(\zeta_0, r)$, with
$r<1$. We write
 $A=fg$, with $$g(z):=\log\left(\frac 4{1-\langle z,\zeta\rangle}\right),$$
where $\zeta:= (1-r)\zeta_0$. The constant $4$ has been chosen in such a way that $g$,
which is holomorphic on $\mathbb{B}^n$,  does not vanish.  We
first remark that we have the required inequality for $f$, that
is,
\begin{equation}\label{logmol}
    \Vert f\Vert_{{\rm mol}(B,L')}\lsim
\frac{\Vert A\Vert_{{\rm
mol}(B,L)}}{\log(e+\sigma(B)^{-1})}
\end{equation}
 for $L'<L$. Indeed, this is a direct consequence of
the two inequalities
$$|g(z)|\gsim \log (4/r)\simeq \log(e+\sigma(B)^{-1})\ \ \ z\in B(\zeta_0, 2r),$$
$$|g(z)|\gsim  \log(e+\sigma(B)^{-1})\left(\frac {r}{d(\zeta_0, z)}\right)^\eps\ \ \ z\notin B(\zeta_0, 2r)$$
for $\eps >0$. We have used  the fact that, for $u>1$ and $ v >e$, one has the inequality
$$\log (uv)\leq 2 u^{\eps} \log v.$$

We now prove that $g$ belongs uniformly to $\BMOA(\bB^n)$ or,
equivalently, that $(1-|z|^2)|\nabla
g|^2\simeq\frac{(1-|z|^2)}{|1-\langle z,\zeta\rangle|^2}$ is
a Carleson measure with uniform bound. Let $B_\rho=B(x_0,\rho)$
be a  ball on the boundary of $\bB^n$ and $T(B_\rho)$ be the
tent over this ball. We have to prove that
$$\int_{T(B_\rho)} \frac{(1-|z|^2)}{|1-\langle z,\zeta\rangle|^2}dV(z)\lsim \sigma(B_\rho)$$
 with constants that are independent of $B_\rho$, $r$ and $\zeta_0$, or, which is equivalent,
$$\int_0^{\rho}\int_{B_\rho} \frac{t}{(d(w,\zeta_0)+t)^2}\,dt\,d\sigma(w)\lsim \sigma(B_\rho).$$
 If $d(x_0,\zeta_0)\ge 2\rho$ then, for
$w\in B_\rho$, we have $d(w,\zeta_0)\ge \rho$ and the denominator is
bounded below by $\rho$, which allows to conclude. When
$d(x_0,\zeta_0)\le 2\rho$, then $B_\rho$ is included in $\widetilde
{B_\rho}:= B(\zeta_0, 3\rho)$ which has a measure comparable to
$B_\rho$. Integrating first in $t$, we have to prove that

$$\int_{\widetilde
{B_\rho}}\log\left(\frac \rho{d(\zeta_0,w)}\right) d\sigma(w)\lsim
\sigma(\widetilde {B_\rho}).$$ To prove this last inequality, we
cut the ball $\widetilde {B_\rho}$ in dyadic shells. We conclude
by using the inequality $$\sum_{j>0}j\sigma(B(\zeta_0,
2^{-j}\rho))\lsim \sigma(B_\rho),$$ which is a consequence of the
fact that
$$\sigma(B(z,2^{-j}\rho))\lsim 2^{-{jn}}\sigma(B(z,\rho)).$$
We have recalled this classical inequality in \eqref{doubling-balls}.

Assume now that  $\Psi(\Vert A\Vert_{{\rm mol}(B,L)})\sigma(B)\leq
1$. We use the fact that $\log t\simeq \log \Psi(t)$ to get
$$ \Vert A\Vert_{{\rm mol}(B,L)}\lsim \log (e+\sigma (B)^{-1})$$
and \eqref{logmol} to conclude that
$$\Phi(\Vert f\Vert_{{\rm mol}(B,L')})\lsim \Psi(\Vert
A\Vert_{{\rm mol}(B,L)}).$$

The weak factorization, that is, Theorem \ref{fact-dec}, follows
directly from Theorem \ref{mol-dec} (molecular decomposition) and
Theorem \ref{factorization} (factorization of molecules), with the
bound below for the quasi-norm of $f$ in the Hardy-Orlicz space.
The bound above uses the direct inequality for molecules, that is
 Proposition \ref{at-gives-mol}, and for products, that is
\eqref{ineq-product}.

We will give some complements to the characterization of symbols
of bounded Hankel operators. If $\exp \H$ denotes the class of
holomorphic functions $f$ such that $f(r\cdot)$ is uniformly in
the exponential class $\exp L$, then Proposition \ref{prod} is
still valid with $BMOA$ replaced by $\exp \H$. Let us remark that
the space $\exp \H$ is the dual space of $P_S (L\log L)$, that is,
the space of functions that may be written as $P_S g$, with $g\in
L\log L$, equipped with the norm
$$\|h\|_{P_S (L\log L)}:=\inf \{\|g\|_{L\log L}\,; h=P_S g\}.$$
Then, looking at the proof of Corollary \ref{Hankel}, we see that
we have as well the following improvement, since $P_S (L\log L)$
is contained in $\H^1$.
\begin{prop}If $b$ belongs to $\BMOA(\varrho_\Psi)$, then $h_b$  extends into a continuous operator from
$\mathcal H^\Phi(\bB^n)$ to $P_S (L\log L)$. \end{prop}

This has been proven by different methods in \cite{BM}.

The same reasoning allows to characterize as well the Hankel
operators which map $\mathcal H^\Phi(\bB^n)$ to $\H^1_{\rm weak}$.
\begin{prop} $h_b$ extends into a continuous operator from
$\mathcal H^\Phi(\bB^n)$ to $\H^1_{\rm weak}$ if and only if $b$
belongs to  $\BMOA(\varrho_\Phi)$.
\end{prop}

The necessity of the condition follows from the fact that
$$|\langle b, f\rangle|=|\int_{\mathbb{S}^n}b\bar f d\sigma|\leq
\|h_b\|\|f\|_{\H_\Phi},$$ so that $b$ defines a continuous linear
form on the space $\H_\Phi$. To prove the sufficiency, it is
sufficient to prove that $h_b$ maps $\mathcal H^\Phi(\bB^n)$ to
$P_S(L^1)$ when $b$ is in $\BMOA(\varrho_\Phi)$. But the dual of $P_S(L^1)$ identifies with $\mathcal H^{\infty}$. So, using
duality, it is sufficient to prove that multiplication by an
element of the dual, that is, $\H^\infty$, maps $\H_\Phi$ into
itself. This is straightforward.

\section{Extension of the results in a general setting}
\label{Casgeneral}

We are now going to give the main points which allow to extend our results in a larger class of domain including strictly pseudoconvex domains and  convex domains of finite type.
Let $\Omega$ be a smooth bounded domain in $\bC^n$. Define the
Hardy-Orlicz spaces as the space of holomorphic functions $f$ so
that
 $$\sup_{0<\eps<\eps_{0}} \int_{\delta(w)=\eps}
{\Phi}(|f|)(w)
\,d\sigma_\eps(w) <\infty
$$
where ${\Phi}$ is as before of lower type $p$, $\delta(w)$ is the distance from $w$ to
$\partial\Omega$ and $d\sigma_\eps$ the Euclidean measure on the level set $\delta(w)=\eps$. Recall that the usual  Hardy space of holomorphic functions $\H^p (\Omega)$ on
$\Omega$ corresponds to the case $\Phi(t)=t^p$.

\subsection{Geometry of H-domains}

\begin{defn}{\rm We say that $\Omega$  is an $H$-domain if it is a smoothly bounded pseudoconvex domain of finite type and if, moreover, for each $\zeta\in\p\Omega$ there exist a neighborhood $V_\zeta$ and a biholomorphic map $\Phi_\zeta$ defined on $V_\zeta$  such that $\Phi_\zeta(\Omega\cap V_\zeta)$ is geometrically convex. }
\end{defn}
We recall  that a point $\zeta\in\p\Omega$ is said to be of finite type if the (normalized) order of contact with $\p\Omega$ of complex varieties at $\zeta$ is finite. By \cite{BoSt} and our assumption it suffices to consider the order of contact of $\p\Omega$ at $\zeta$ with $1$-dimensional complex manifolds, see\cite{BoSt} and references therein. The domain $\Omega$ is said to be of finite type if every point on $\p\Omega$ is of finite type.We denote by $M_\Omega$ the maximum of the types of points on $\p\Omega$. Notice that the class of $H$-domains contains both the convex domains of finite type   and the strictly pseudoconvex domains.
\medskip

We describe the geometry of an $H$-domain $\Omega$.This is done locally, using a partition of unity. Moreover, in a neighborhood of a point $\zeta\in\p\Omega$, using local coordinates and the assumption, we may in fact assume that $\Omega$ is geometrically convex. Thus, we do not lose generality if we assume that it is globally convex. Then, there exist an $\eps_0>0$ and a defining function $\varrho$ for $\Omega$ such that for $-\eps_0<\eps<\eps_0$ the sets $\Omega_\eps:= \{z\in\bC^n: \varrho(z)<\eps\}$ are all convex.  Moreover, denote by $U=U_{\eps_0}$ the tubular neighborhood of $\p\Omega$ given by $\{z\in\bC^n: -\eps_0<\varrho(z)<\eps_0\}$.  By taking $\eps_0>0$ sufficiently small, we may assume that on $\ov{U}$ the normal projection $\pi$ of $U$ onto $\p\Omega$ is uniquely defined. Let $z\in U$ and let $v$ be a unit vector in $\bC^n$.  We denote by $\tau(z,v,r )$ the distance from $z$ to the surface $\{z':\varrho(z')=\varrho(z)+r \}$ along the complex line determined by $v$. One of the basic relations among the quantities defined above is the following.There exists a constant $C$ depending only on the geometry of the domain such that given $z\in U$, any unit vector $v\in\bC^n$ that is orthogonal to the level set of the function $\varrho$ and $r \le r_0$  and $\eta<1$ we have
\begin{equation}\label{dilating}
 C^{-1}\eta^{1/2}\tau(z,v,r)\leq \tau(z,v,\eta r)
\leq C\eta^{1/M_\Omega}\tau(z,v,r) .
\end{equation}
We next define the $r$-{\em extremal orthonormal basis} $\{v^{(1)},\dots,v^{(n)}\}$ at $z$, which generalize the choices that we have done for the unit ball. The first vector is given by the direction transversal direction to the level set of $\varrho$ containing $z$, pointing outward. In the complex directions orthogonal to $v^{(1)}$ we choose $v^{(2)}$ in such a way that $\tau(z,v^{(2)},r )$ is maximum. We repeat the same procedure to determine the remaining elements of the basis. We set
$$\tau_j (z,r ) = \tau(z,v^{(j)},r ).$$
By definition, $\tau_1(z, r)\simeq r$.
The polydisc $Q(z,r )$ is now given as
$$Q(z,r ) = \{w: |w_k|\le \tau_k(z,r ),\, k=1,\dots,n\}.$$
Here  $(w_1,\dots,w_n)$ are the coordinates determined by $r $-extremal orthonormal basis $\{v^{(1)},\dots,v^{(n)}\}$ at $z$. Notice that these coordinates $(w_1,\dots,w_n)=(w_1^{z,r },\dots,w_n^{z,r })$ depend on $z$ and on  $r $.  They are called {\sl special coordinates at the point $z$ and at scale $r$}. The quasi-distance is  defined by setting
\begin{equation}\label{quasi-distance}
d_b(z,w) = \inf\{ r : w\in Q(z,r )\}.
\end{equation}

Notice that by the above properties the sets $Q(z,r )$ are in fact equivalent to the balls in the quasi-distance $d_b$. We also consider balls on the boundary $\p\Omega$ defined in terms of $d_b$. For $\zeta\in \p\Omega$ and $r >0$ we set $$ B(\zeta,r)=\{z\in \p\Omega: d_b(z,\zeta)<r \}. $$ These balls are equivalent to the sets $Q(\zeta,r )\cap\p\Omega$. Moreover, we define the function $d$ on $\ov{\Omega}\times\ov{\Omega}$ by setting
\begin{equation}\label{defn-d}
d(z,w)=\delta(z)+\delta(w)+d_b\bigl(\pi(z),\pi(w)\bigr)\, ,
\end{equation}
where $\pi$ is the normal projection of a point $z$ onto the boundary. We now set
$$ {\bf \tau}(z,r)= ( \tau_1(z,r)\cdots, \tau_n(z,r)). $$
 Then, for $\alpha$ a multiindex, we note
 $$ {\bf \tau}^{\alpha}(z,r)= \prod_{j=1}^{n} \tau_j^{\alpha_j}(z,r). $$When $\Omega$ is strictly pseudoconvex, we have simply  ${\bf \tau}^{\alpha}(z,r)\simeq r^{\frac{|\alpha|+\alpha_1}2}$.
\medskip
Let  $\sigma$ denotes the surface measure on $\p\Omega$.
Then, one has
$$ \sigma\bigl(B(w,r)\bigr) \simeq {\bf \tau}^{(1,2, \cdots, 2)}(w,r). $$
Moreover, the property \eqref{doubling-balls} is replaced by the double inequality
\begin{equation}\label{doubling-balls-finite}
       \lambda^n \sigma(\zeta_0, r)\lsim \sigma(B(\zeta_0,\lambda r))\lsim \lambda^{1+(2n-2)/M_\Omega}  \sigma(B(\zeta_0,\lambda r)),
\end{equation}
As we said before, all these definitions are local, and may be given in the context of $H$-domains.

As in the case of the unit ball, if $w_j$ are the coordinates of $w-z$ in the basis  $\{v^{(1)},\dots,v^{(n)}\}$ and if $w_j=s_j+it_j$, then $s_j$ for $j\geq 2$ and $t_j$ for $j\geq 1$ define $2n-1$ local coordinates of $\p\Omega$ in a neighborhood of $z$. We will still speak of {\sl special coordinates at the point $z$ and the scale $r$}.
\medskip

In the neighborhood of $z\in \partial \Omega$, the hypersurface $\partial \Omega$ coincides with the graph $\Re w_1=h(\Im w_1, w')$, with $w'=(w_2, \cdots, w_n)$. As in the case of the unit ball,  we are interested in estimates on $D^{(\alpha, \beta)}_wS(\, (h(t_1, s'+it')+it_1, w'))$, where $\alpha$ is an $n-1$-index of derivation in the variable $s'$, while $\beta$ is an $n$-index of derivation in $t$. The equivalent of \eqref{estimateSzego} is given by the estimates of McNeal and Stein \cite {McS1} and \cite {McS2}, see also \cite{BPS}, Lemma 4.7 for an analogous context.
For $d(w, z)<r$ and $\zeta \notin B(z, Cr)$, we have
 \begin{equation}\label{estimateSzegofinite}|D^{(\alpha, \beta)}_wS(\zeta, (h(t_1, s'+it')+it_1, w'))|\lsim {\bf \tau}^{-(1+\beta_1,2+\alpha_2+\beta_2, \cdots, 2+\alpha_n+\beta_n)}(z, d(w, z)).  \end{equation}

As in \cite{BPS}, we will also use the existence of a support function given by Diederich and Fornaess \cite{DF}.
\begin{thm}\label{supp-funct} Let $\Omega$ be a smoothly bounded pseudoconvex $H$-domain of finite type in $\bC^n$.  Then there exist a neighborhood $U$ of the boundary $\p\Omega$ and a function $H\in{\mathcal C}^\infty(\bC^n\times U)$ such that the following conditions hold:
\begin{itemize}
\item[$(i)$] $H(\cdot,w)$ is holomorphic on $\Omega$ for all $\zeta\in U$;
\item[$(ii)$] there exists a constant $c_1>1$ such that
$$\frac{1}{c_1} d(z,w) \le|H(z,w)|\le c_1 d(z,w).$$
\end{itemize}
\end{thm}

With all these definitions, we claim the following.

\noindent {\bf Statement of results for $H$-domains.}  {\sl The analogues of Theorems \ref{max-charact} to Corollary \ref{Hankel} are valid for the $H$-domain $\Omega$ with the following modifications:  $N_p:=(\frac 1p-1)(M_\Omega+2n-2)-1$ in Definition \ref{realHardy-Orlicz}; in Proposition \ref{molecule-property}, the condition is $L<\frac{2N+2}{M_\Omega}$, while in Proposition \ref{at-gives-mol}, we have $L_p:=(2/p-2)\left(1+\frac{2n-2}{M_\Omega}\right)$. Finally, for the definition of $\BMO(\varrho)$, we have to take $N+1>\alpha(M_\Omega+2n-2)$.}

Let us sketch the modifications to be done. Atoms adapted to a ball $B:=B(\zeta_0, r_0)$ are defined as before,  using special coordinates at $\zeta_0$ and at scale $r_0$ to define the vanishing moment conditions. Remark that the coordinates depend on $r_0$, but the space $\mathcal P_N(\zeta_0)$ does not.

Then, in Lemma \ref{estimA}, the second estimate has to be replaced by
\begin{equation}\label{farfinite}
  |A(\zeta)| \lsim \biggl(\frac{r_0}{d(\zeta,\zeta_0)}\biggr)^{
\frac{N+1}{M_\Omega}} \frac{\Vert a\Vert_2\sigma(B)^{\frac 12}}{\sigma(\zeta_0, d(\zeta, \zeta_0))} \  \ \mbox{\rm for} \
d(\zeta,\zeta_0)\ge Cr_0.
\end{equation}
The proof is the same, using the estimates \eqref{estimateSzegofinite} in place of \eqref{estimateSzego}.
\medskip

Next, molecules are defined as follows.
\begin{defn}\label{moleculefinite}
A holomorphic function $A\in \H^2(\Omega)$ is called a molecule of
order $L$, associated to the ball $B:=B(z_0,r_0)\subset \partial \Omega$, if
it satisfies
\begin{equation}\label{molfinite}
\sup_{\eps<\eps_0}\int_{\partial \Omega}\left(1\,+\,\frac{d(z_0,\xi)^{L}}{r_0^{L}}\times \frac{\sigma(B(z_0, d(z_0,\xi)))}{\sigma(B(z_0, r_0))}\right)|A(\xi-\eps \nu(\xi))|^2
\frac{d\sigma(\xi)}{\sigma(B)}<\infty,
\end{equation}
\end{defn}
with $\nu$ the outward normal vector. In this case, the left hand side is $\Vert A\Vert_{{\rm
mol}(B,L)}^2$. It follows  from \eqref{farfinite}, cutting the integral into dyadic balls, that the projection of an atom related to the ball $B:=B(z_0,r_0)\subset \partial \Omega$ is a molecule of order $L<\frac{2N+2}{M_\Omega}$.

Finally, to see that a molecule of order $L$ is in the Hardy space $\H^{\Phi}$, we  prove that, with $B_k:=B(z_0,2^k r_0)$,
$$\int_{B_k\setminus B_{k-1}}\Phi(g) \frac{d\sigma}{\sigma(B)}\lsim 2^{-k\varepsilon}\Phi
\left(\left(2^{kL}\frac{\sigma(B_k)}{\sigma(B)}\int_{B_k\setminus B_{k-1}}
g(\xi)^2\frac{d\sigma(\xi)}{\sigma(B)}\right)^{1/2}\right)$$ for
some $\varepsilon
>0$. To prove this last inequality, we use again Jensen Inequality
\eqref{jensen} for the measure $d\sigma$ on $B_k$, divided by its total mass $\sigma(B_k)$. This gives
$$\int_{B_k\setminus B_{k-1}}\Phi(g) \frac{d\sigma}{\sigma(B)}\lsim \frac{\sigma(B_k)}{\sigma(B)}\Phi
\left(\left(\int_{B_k\setminus B_{k-1}}
g(\xi)^2\frac{d\sigma(\xi)}{\sigma(B_k)}\right)^{1/2}\right).$$ We
conclude by using the fact that $\Phi$ is of lower type $p$, which
allows to write that $$\frac{\sigma(B_k)}{\sigma(B)}\Phi(t)\lsim \Phi\left(\left (\frac{\sigma(B_k)}{\sigma(B)}\right)^{\frac 1p}t\right).$$  Using \eqref{doubling-balls-finite}, one finds that it is
sufficient to choose $L>L_p:=2(1/p-1)\left(1+\frac{2n-2}{M_\Omega}\right)$.

\medskip

Up to now, we have given the modifications for having the atomic decomposition, the continuity of the Szeg\"o projection, the duality. It remains to see the modifications in the proof of the factorization theorem. As at the beginning of Section \ref{Sectionfactorization},  we factorize each
molecule $A$ associated to a ball $B:=B(\zeta_0, r)$ as $B=fg$, with $B$ a molecule and $g$ a $BMOA$-function.

For this  factorization, we use the support function given in Theorem \ref{supp-funct}.
We set
$H_0=H(\cdot,\tilde \zeta_0)$, where $\tilde
\zeta_0=\zeta_0-r\nu(\zeta_0)$.  We choose $g=\log(cH_0^{-1})$ with $c$ so that $g$ does not vanish in $\Omega$.

We have as before the inequality
\begin{equation}\label{logmol-finite}
    \Vert f\Vert_{{\rm mol}(B,L')}\lsim
\frac{\Vert A\Vert_{{\rm
mol}(B,L)}}{\log(e+\sigma(B)^{-1})}
\end{equation}
 for $L'<L$. Just use (ii) in Theorem \ref{supp-funct}.

We now prove that $\log(cH_0^{-1})$ belongs  to $BMOA$ with bounds independent of $\zeta_0$ and $r$. The proof follows the same line as the one in the unit ball, using that $|H_0|$ and $|\nabla H_0$ are uniformly bounded in $\Omega$.

This finishes the proof of the factorization theorem.

\epf

\end{document}